\documentclass[reqno, 12pt]{amsart}
\usepackage{amsmath}
\usepackage{amssymb}
\usepackage{amsthm}

\usepackage[a4paper]{geometry}

\newcommand\NoBlackBoxes{\global\overfullrule0pt}
\NoBlackBoxes
\theoremstyle{plain} 

\newcommand{\IR}{\mathbb{R}}
\newcommand{\IN}{\mathbb{N}}

\newtheorem{theorem}{Theorem}[section]
\newtheorem{lemma}[theorem]{Lemma}
\newtheorem{corollary}[theorem]{Corollary}
\newtheorem{proposition}[theorem]{Proposition}

\theoremstyle{definition} 

\newcommand{\abs}[1]{\left\lvert #1 \right\rvert}  
\newcommand{\bigabs}[1]{\bigl\lvert #1 \bigr\rvert}  
\newcommand{\Bigabs}[1]{\Bigl\lvert #1 \Bigr\rvert}  
\newcommand{\norm}[1]{\left\lVert #1 \right\rVert}

\numberwithin{equation}{section}

\begin{document}
\title{Freeness of linear and quadratic forms 
in von Neumann algebras.} 
\author[G. P. Chistyakov]{G. P. Chistyakov$^{1,2,4}$}
\thanks{1) Faculty of Mathematics, 
University of Bielefeld, Germany.} 
\thanks{2) Institute for Low Temperature Physics and Engineering, 
Kharkov, Ukraine}
\thanks{3) Institut f\"ur Mathematische Strukturtheorie, Technische Universit\"at Graz}
\thanks{4) Research supported by SFB 701.}
\address{Gennadii Chistyakov \newline
Fakult\"at f\"ur Mathematik\newline
Universit\"at Bielefeld\newline
Postfach 100131\newline
33501 Bielefeld \newline
Germany}
\email {chistyak@math.uni-bielefeld.de, \quad chistyakov@ilt.kharkov.ua} 
\author[F. G\"otze]{F. G\"otze$^{1,4}$}
\address
{Friedrich G\"otze\newline
Fakult\"at f\"ur Mathematik\newline
Universit\"at Bielefeld\newline
Postfach 100131\newline
33501 Bielefeld \newline
Germany}
\email {goetze@math.uni-bielefeld.de}
\author[F. Lehner]{F. Lehner$^{3,4}$}
\address
{Franz Lehner \newline
Institut f\"ur Mathematische Strukturtheorie\newline 
Technische Universit\"at Graz\newline 
Steyrergasse 30, A-8010 Graz\newline  
Austria}  
\email{lehner@math.tu-graz.ac.at} 
\date{\today{}}
\markboth{Free linear and quadratic forms}{G. P. Chistyakov, 
F. G\"otze and F.Lehner}
\leftmark
\rightmark
\begin{abstract}
  We characterize the semicircular distribution by freeness 
  of linear and quadratic forms in noncommutative random 
  variables from tracial $W^*$-probability spaces with 
  relaxed moment conditions.
\end{abstract}
\maketitle
\section{Introduction} 
The~intensive research in the~asymptotic theory of random matrices
has motivated increased research on infinitely dimensional limiting models. 
Free convolution of probability measures, 
introduced by D. Voiculescu, may be regarded as such a model~\cite{Vo:1986},
\cite{Vo:1987}.
The~key concept of this definition is the~notion of freeness,
which can be interpreted as a~kind of independence for noncommutative
random variables. As in classical probability the~concept of
independence gives rise to classical convolution, the~concept
of freeness leads to a~binary operation on probability measures
on the~real line which is called free convolution. Many classical results
in the~theory of addition of independent random variables have their
counterpart in this theory, such as the~law of large numbers,
the~central limit theorem, the~L\'evy-Khintchine formula and others.
We refer to Voiculescu, Dykema and Nica~\cite{Vo:1992}, Hiai and Petz~\cite{HiPe:2000},
and Nica and Speicher~\cite{NiSp:2006} for an~introduction to these topics. 

The~central limit theorem for free random variables holds with limit distribution
equal to a semicircle law. Semicircle laws play in many respects the role of Gaussian laws, 
when independence is replaced by freeness in a noncommutative probability space.  

In usual probability theory various characterizations of the Gaussian law have been 
obtained, for instance see~\cite{KLR:1973}. In particular, there is the well-known fact that 
the independence of the sample mean and the simple variance of independent identically 
distributed random variables characterizes the Gaussian laws, see~\cite{Z:1951} and \cite{KS:1949}.

Hiwatashi, Nagisa and Yoshida~\cite{HiNaYo:1999} established the characterization
of the semicircle law by freeness of a certain pair of a linear and a quadratic form 
in free identically distributed bounded noncommutative random variables, which covers
the free analogue of the previous result in usual probability theory. 

In this paper we generalize the Hiwatashi, Nagisa and Yoshida result 
to the case of not necessarily bounded identically distributed
noncommutative random variables
requiring only finiteness of the second moment.

Unbounded operators affiliated to a von Neumann algebra play the role of unbounded
measurable random variables in noncommutative probability. A general theory of such
operators has been developed already by Murray and Neumann~\cite{MuNe:1936}. In free probability
unbounded random variables have so far only been considered by Maassen~\cite{Ma:1992} from 
the analytic point of view and by Bercovici and Voiculescu~\cite{BeVo:1993} in great detail.

The plan of the present paper is as follows.
In Section~2 we formulate our results. In Section~3 we give auxiliary
results on measurable operators. In Section~4 we prove auxiliary analytic
results. Finally in Section~5 we prove our main result
by carefully adapting classical moment estimates to the noncommutative
situation.


\emph{Acknowledgements}. The authors thank the anonymous referee
for numerous remarks, in particular pointing out an error in
Proposition~\ref{pr2.2}.

\section{Results}
Assume that $\mathcal A$ is a~finite von Neumann algebra with normal faithful
trace state $\tau$  acting on a~Hilbert space $H$. 
The~pair $(\mathcal A,\tau)$ will be called a~{\it tracial}
$W^*$-{\it probability space}. 
We will denote  by $\tilde {\mathcal A}$ the~set of all operators
on $H$ which are affiliated with $\mathcal A$ and by 
$\tilde{\mathcal A}_{sa}$ its real subspace
of selfadjoint operators. 
Recall that a (generally unbounded) selfadjoint
operator $X$ on $H$ is affiliated with $\mathcal A$ if all 
the~spectral projections
of $X$ belong to $\mathcal A$. The~elements of $\tilde{\mathcal A}_{sa}$ 
will be regarded as (possibly) unbounded random variables. 
The set $\tilde{\mathcal A}$ is actually an algebra,
as shown by Murray and von Neumann~\cite{MuNe:1936}, and  the usual problems
concerning domains of definition are settled once  for all.
The~distribution $\mu_{T}$ of an element $T\in \tilde{\mathcal A}_{sa}$
is the~unique  probability measure 
on $\IR$ satisfying the~equality
$$
\tau(u(T))=\int\limits_{\mathbb R}u(\lambda)\,\mu_T(d\lambda)
$$
for every bounded Borel function $u$ on $\mathbb R$. 

A~family $(T_j)_{j\in I}$ of elements of 
$T\in \tilde{\mathcal A}_{sa}$ 
is said
to be \emph{free} if for all bounded continuous functions $u_1,u_2,\dots,u_n$
on $\mathbb R$ we have $\tau(u_1(T_{j_1})u_2(T_{j_2})\dots u_n(T_{j_n}))=0$
whenever $\tau(u_l(T_{j_l}))=0,\,l=1,\dots,n$, 
for every choice of alternating indices
$j_1,j_2,\dots,j_n$. 

Denote by $\mu_{w_{m,r}}$ the~semicircle distribution with density 
$\frac 2{\pi r^2}\sqrt{(r^2-(x-m)^2)_+}$,
where $m\in\mathbb R,r\in\mathbb R_+$ and $a_+:=\max\{a,0\}$ 
for $a\in\mathbb R$. This distribution plays the~role of 
Gaussian one, when independence is replaced by freeness. 

The~main aim of this note is to prove the~following characterization theorem.

\begin{theorem}\label{th2.1}
Let $T_1,T_2,\dots,T_n$ be free identically distributed random variables 
with zero expectations,
$\tau(T_j)=0$, and $\tau(T_j^2)<\infty$ in $W^*$-probability 
space $(\mathcal A,\tau)$. 
Let $A=(a_{ij})\in M_n(\mathbb R)$ be an~$n\times n$ symmetric 
real matrix 
and $\mathbf b=\sideset{^t}{}{\mathop{(}}b_1,b_2,\dots,b_n)\in\mathbb R^n$ 
be an $n$-dimensional vector satisfying the~conditions
\begin{equation}\label{re.1}
A\mathbf b=\mathbf 0\quad\text{and}\quad \sum_{j=1}^nb_j^ma_{jj}\ne 0
\quad\text{for}\quad m\in\mathbb N.
\end{equation}
Then the~linear form $L=\sum_{j=1}^nb_jT_j$ and the~quadratic 
form $Q=\sum_{j,k}^na_{jk}T_jT_k$
are free if and only if $T_1$ has semicircle distribution. 
\end{theorem}

\begin{corollary}\label{cor2.1}
Let $T_1,T_2,\dots,T_n$ be free identically distributed random variables 
with zero expectations,
$\tau(T_j)=0$, and $\tau(T_j^2)<\infty$ in $W^*$-probability 
space $(\mathcal A,\tau)$. 
Then the~sample mean $\overline{T}=\frac 1n\sum_{j=1}^nT_j$ and the~sample variance 
$V=\frac 1n\sum_{j=1}^n(T_j-\overline{T})^2$ are free if and only 
if $T_1$ has semicircle distribution.  
\end{corollary}

Theorem~\ref{th2.1} and Corollary~\ref{cor2.1} for bounded free 
identically distributed random variables under the assumptions that
$A$ is non-negative definite and $\mathbf b$ is non-negative 
was proved by Hiwatashi, Nagisa and Yoshida~\cite{HiNaYo:1999}. 
A more general version of Theorem~\ref{th2.1} for bounded free identically distributed 
random variables was proved by the last author in~\cite{Le:2003}. 
Therefore we only need to prove the
``only if'' part of Theorem~\ref{th2.1}. In order to do this  
we establish that the~freeness of $L$ and $Q$ 
implies that the distribution of $T_1$ has moments of all order, i.e.,
$\tau(\abs{T_1}^k)<\infty,\,k\in\IN$, where $\abs{T}=(T^*T)^{1/2}$.
Namely, we prove the following result.
\begin{theorem}\label{th2.2}
Let $T_1,T_2,\dots,T_n$ be free identically distributed random variables 
in $W^*$-probability space $(\mathcal A,\tau)$
such that $\tau(T_j^2)<\infty$. 
We consider the~linear form $L=\sum_{j=1}^nb_jT_j$ and the~quadratic 
form $Q=\sum_{j,k}^na_{jk}T_jT_k$ with real coefficients
$b_j$ and $a_{jk}$ 
such that
\begin{equation}\label{5.1} 
b_ja_{jj}\ne 0 \quad \text{for some}\quad j\in\{1,2,\dots, n\}.
\end{equation}
If the~forms $L$ and $Q$ are free, then $\tau(\abs{T_1}^k)<\infty,\,k=1,2,\dots$.
\end{theorem}

In particular, we infer from this result that under very weak assumptions
freeness of linear and quadratic forms in noncommutative random 
variables from a~tracial $W^*$-probability space
automatically implies finiteness of all moments.


\section{Auxiliary results. Measurable operators and 
integral for a trace}

We fix a~faithful finite normal trace $\tau$ on a~finite von Neumann algebra
$\mathcal A$. By $\tilde{\mathcal A}$ we denote the~completion of $\mathcal A$
with respect to $\tau$-measure topology. We denote 
$\tilde{\mathcal A}_+=\{a^*a:a\in\tilde{\mathcal A}\}$ as well.
The~function $\tau$ on $\tilde{\mathcal A}_+$ enjoys 
the~following properties (see~\cite{Ta:2003}, p. 176):
\begin{align}
\tau(a+b)=\tau(a)&+\tau(b),\quad a,b\in\tilde{\mathcal A}_+,\quad
\tau(\lambda a)=\lambda\,\tau(a),\quad \lambda\ge 0;\notag\\
&\tau(x^*x)=\tau(xx^*),\quad x\in\tilde{\mathcal A}.\notag
\end{align}

For $1\le p<\infty$, set
\begin{equation}\notag
\norm{x}_p=\tau(\abs{x}^p)^{1/p},\quad x\in \tilde{\mathcal A};\quad
L^p(\mathcal A,\tau)=\{x\in \tilde{\mathcal A}:\, \norm{x}_p<\infty\}.
\end{equation}
Then $L^p(\mathcal A,\tau)$ is a~Banach space in which 
$\mathcal A\cap L^p(\mathcal A,\tau)$
is dense. Furthermore, $L^p(\mathcal A,\tau)$ is a two-sided operator ideal
and 
\begin{equation}\label{au1.1}
\norm{ax}_p\le\norm{a}\norm{x}_p,\quad \norm{xa}_p\le\norm{a}\norm{x}_p
\end{equation}
for each $a\in\mathcal A,\,x\in L^p(\mathcal A,\tau)$.

If $1/p_1+\dots+1/p_n=1$ and $p_j>1,
j=1,\dots,n$, then the~product of $L^{p_1}(\mathcal A,\tau),\dots$, 
$L^{p_n}(\mathcal A,\tau)$ 
coincides with $L^1(\mathcal A,\tau)$ and we have the~H\"older inequality:
\begin{equation}\label{au1.2} 
\abs{\tau(x_1x_2\dotsm x_n)}\le \norm{x_1}_{p_1}\norm{x_2}_{p_2}\dotsm \norm{x_n}_{p_n},
\quad x_1\in L^{p_1}(\mathcal A,\tau),\dots,
\quad x_n\in L^{p_n}(\mathcal A,\tau).
\end{equation}

Since $x_1x_2\dotsm x_n$ admits a~representation $x_1x_2\dotsm x_n=u\abs{x_1x_2\dotsm x_n}$, 
where $u\in \mathcal A$ is a~partial isometry, we have, using
(\ref{au1.1}) and (\ref{au1.2}),
\begin{align}\label{au1.3} 
\norm{x_1x_2\dotsm x_n}_1&=\tau(\abs{x_1x_2\dotsm x_n})=\tau(u^*x_1x_2\dotsm x_n)
\le\norm{u^*x_1}_{p_1} \norm{x_2}_{p_2}\dotsm \norm{x_n}_{p_n}\notag\\
&\le \norm{u^*}_{\infty}\,\norm{x_1}_{p_1}\norm{x_2}_{p_2}\dotsm\norm{x_n}_{p_n}=
\norm{x_1}_{p_1}\norm{x_2}_{p_2}\dotsm\norm{x_n}_{p_n}. 
\end{align}

In later reference we state the noncommutative \emph{Minkowski inequality}
\begin{equation}\label{au1.4}
\norm{x_1+\dots+x_n}_p\le \norm{x_1}_p+\dots+\norm{x_n}_p
\end{equation}
for $1\le p<\infty$.

\section{Auxiliary analytic results}

Denote by $\bf M$ the~family of all Borel probability measures
on the~real line $\mathbb R$.

Let $T_1$ and $T_2$ are free random variables with distributions
$\mu_1$ and $\mu_2$ from $\bf M$, respectively. Following Bercovici and
Voiculescu \cite{BeVo:1993} we define the additive free convolution $\mu_1\boxplus\mu_2$ as 
the distribution of $T_1+T_2$. 

Let $\bf M_+$ be the~set of probability measures  $\mu$ on 
$\mathbb R_+=[0,+\infty)$ such that $\mu(\{0\})<1$.

Fix probability measures $\mu_1,\mu_2\in\bf M_+$ and
fix random variables $T_j$ such that their disributions
$\mu_{T_j}=\mu_j$.
Following~\cite{BeVo:1993} we set $\mu_1\boxtimes\mu_2=
\mu_{T_1^{1/2}T_2T_1^{1/2}}=\mu_{T_2^{1/2}T_1T_2^{1/2}}$.

Define, following 
Voiculescu~\cite{Vo:1987}, the~$\psi_{\mu}$-function of 
a~probability measure $\mu\in\bf M_+$, by
\begin{equation}\label{2.4}
\psi_{\mu}(z)=\int\limits_{\mathbb R_+}\frac{z\xi}{1-z\xi}\,\mu(d\xi)
\end{equation}
for $z\in\mathbb C\setminus\mathbb R_+$. The~measure $\mu$ is completely
determined by $\psi_{\mu}$. 
Note that $\psi_{\mu}:\mathbb C\setminus\mathbb R_+\to \mathbb C$ is 
an~analytic function such that $\psi_{\mu}(\bar z)
=\overline{\psi_{\mu}(z)}$, and $z(\psi_{\mu}(z)+1)\in\mathbb C^+$
for $z\in\mathbb C^+$. Consider the~function 
\begin{equation}\label{2.5}
K_{\mu}(z):=\psi_{\mu}(z)/(1+\psi_{\mu}(z)),
\quad z\in\mathbb C\setminus\mathbb R_+.
\end{equation}  
It is easy to see that $K_{\mu}(z)\in\mathcal K$, where 
$\mathcal K$ is the~subclass of $\mathcal N$ of functions $f$ such that 
$f(z)$ is analytic and nonpositive on the~negative real axis, and
$f(-x)\to 0$ as $x\downarrow 0$. 

This subclass of $\mathcal N$ was described by M. Krein~\cite{Kr:1977},
therefore we denote it by $\mathcal K$.


\begin{theorem}\label{th2.3}
There exist two uniquely determined functions $Z_1(z)$ and $Z_2(z)$ 
in the~Krein class $\mathcal K$ such that 
\begin{equation}\label{2.8}
Z_1(z)\,Z_2(z)=zK_{\mu_1}(Z_1(z))\quad\text{and}\quad 
K_{\mu_1}(Z_1(z))=K_{\mu_2}(Z_2(z)),
\quad z\in\mathbb C^+. 
\end{equation}
Moreover $K_{\mu_1\boxtimes\mu_2}=K_{\mu_1}(Z_1(z))$.
\end{theorem}

This result was proved by Biane~\cite{Bi:1998}.  
Belinschi and Bercovici~\cite{BelBe:2007} and Chistyakov
and G\"otze~\cite{ChG:2005} proved this theorem by purely
analytic methods.

For a probability measure $\mu\in\bf M$, define its \emph{absolute moment} of order $\alpha$
\begin{equation}\notag
\rho_{\alpha}(\mu):=\int\limits_{\mathbb R}\abs{x}^{\alpha}\,\mu(dx)
\end{equation}
and for $\mu\in\bf M_+$, define 
\begin{equation}\notag
m_{\alpha}(\mu):=\int\limits_{\mathbb R_+}x^{\alpha}\,\mu(dx),
\end{equation}
where $\alpha\ge 0$.

We now characterize existence of moments in terms of Taylor expansions
of the Krein function.
A similar result for the $R$-transform was obtained by
Benaych-Georges~\cite{Benaych:2006:Taylor} and applied
to additive free infinite divisibility.
\begin{proposition}\label{pr2.1}
Let $\mu\in\bf M_+$. In order that $m_p(\mu)<\infty$ for some $p\in\mathbb N$
it is necessary and sufficient that 
the Krein function $(\ref{2.5})$ admits the expansion 
\begin{equation}\label{pr1a}
\frac 1x K_{\mu}(-x)=-r_1(\mu)+r_2(\mu)x+\dots+(-1)^pr_p(\mu)x^{p-1}+o(x^{p-1})
\quad\text{for}\quad x>0
\quad\text{and}\quad x\downarrow 0,
\end{equation}
with some real coefficients $r_1(\mu),r_2(\mu),\dots,r_p(\mu)$.
\end{proposition}

The~coefficients $r_1(\mu),r_2(\mu),\dots$ coincide with the so-called boolean cumulants,
see Speicher and Woroudi~\cite{SpWo:1997}.
Note that $r_k(\mu)$ depends on $m_1(\mu),m_2(\mu),\dots,m_k(\mu)$ only. 
\begin{proof}
{\it Necessity}. Assume that $m_p(\mu)<\infty$. Then we see that, for $x>0$,
\begin{align}
\psi_{\mu}(-x)+1&=
\frac 1x\int\limits_{\mathbb R_+}\frac{\mu(du)}
{\frac 1x+u}\notag\\
&=\frac 1x\Big(x-m_1(\mu)x^2+\dots+(-1)^p m_p(\mu)x^{p+1}
+(-1)^{p+1}x^{p+1}\int\limits_{\mathbb R_+}\frac{u^{p+1}\mu(du)}
{\frac 1x+u}\Big),\label{pr1}
\end{align}
where
\begin{equation}
\int\limits_{\mathbb R_+}\frac{u^{p+1}\,\mu(du)}{\frac 1x+u}\to 0
\quad\text{as}\quad x\to 0.\label{pr1*}
\end{equation}

By (\ref{pr1}), we have the~relation, for the~same $x$,
\begin{align}
K_{\mu}(-x)&=\frac{\psi_{\mu}(-x)}{\psi_{\mu}(-x)+1}=\psi_{\mu}(-x)
-\psi_{\mu}^2(-x)+\dots+(-1)^{p-1}\psi_{\mu}^p(-x)+O(x^{p+1})\notag\\
&=-r_1(\mu)x+r_2(\mu)x^2+\dots+(-1)^pr_p(\mu)x^{p}+
(-1)^{p+1}x^{p}\int\limits_{\mathbb R_+}\frac{u^{p+1}\mu(du)}
{\frac 1x+u}+O(x^{p+1}).\label{pr2}
\end{align}
Now \eqref{pr1*} and \eqref{pr2} imply the~necessity of the~assumptions 
of Proposition~\ref{pr2.1}.

{\it Sufficiency}. Note that, for positive sufficiently small $0<x\le x_0$, 
\begin{equation}\notag
-\frac 1x K_{\mu}(-x)=\frac 1{\psi_{\mu}(-x)+1}\int\limits_{\mathbb R_+}
\frac {u\,\mu(du)}{1+ux}
\ge \frac 1{\psi_{\mu}(-x)+1}\int\limits_{[0,1/x)}\frac u2\,\mu(du)\ge \frac 12
\int\limits_{[0,1/x)}u\,\mu(du).
\end{equation}
By (\ref{pr1a}), we conclude that $m_1(\mu)<\infty$. 
Assume that the~inequality $m_k(\mu)<\infty$ holds for any
$k\le p-1$. From (\ref{pr1}) we obtain the~formula
\begin{equation}\notag
\psi_{\mu}(-x)=-m_1(\mu)x+\dots+(-1)^{k} m_k(\mu)x^{k}
+(-1)^{k+1}x^{k}\int\limits_{\mathbb R_+}\frac{u^{k+1}\,\mu(du)}
{\frac 1x+u},\quad x>0.
\end{equation}
Using this formula and (\ref{pr2}) with $p=k$ we note that, for small $x>0$,
\begin{align}
(-1)^{k+1}\Big(K_{\mu}(-x)+r_1(\mu)x-r_2(\mu)x^2&-\dots
-(-1)^kr_k(\mu)x^{k}\Big)=x^{k}\int\limits_{\mathbb R_+}\frac{u^{k+1}\,
\mu(du)}{\frac 1x+u}+O(x^{k+1})
\notag\\
&\ge \frac 12x^{k+1}\int\limits_{[0,1/x)}u^{k+1}\,\mu(du)+O(x^{k+1}).\label{pr3}
\end{align}
On the other hand, by (\ref{pr1a}) with $p=k+1$, we have, for small $x>0$, 
\begin{equation}\notag
K_{\mu}(-x)+r_1(\mu)x-r_2(\mu)x^2-\dots-(-1)^kr_k(\mu)x^{k}
=(-1)^{k+1}r_{k+1}(\mu)x^{k+1}+o(x^{k+1}). 
\end{equation}
Therefore we easily conclude from (\ref{pr3}) that $m_{k+1}(\mu)<\infty$. 
Thus induction may be used and the~sufficiency of the~assumptions 
of Proposition~\ref{pr2.1} is also proved.
\end{proof}

Speicher and Woroudi~\cite{SpWo:1997} indicated a~universal formula for calculation of boolean cumulants
$r_k(\mu)$. For example
\begin{align}
&r_1(\mu)=m_1(\mu),\,
r_2(\mu)=m_2(\mu)-m_1^2(\mu),\,r_3(\mu)=m_3(\mu)-2m_1(\mu)m_2(\mu)
+m_1^3(\mu),\notag\\
&r_4(\mu)=m_4(\mu)-m_2^2(\mu)-2m_1(\mu)m_3(\mu)
+3m_1(\mu)^2m_2(\mu)-m_1^4(\mu).\label{pr3*}
\end{align}


\begin{proposition}\label{pr2.2}
Let $\mu\in\bf M_+$ and $\alpha\in(0,1)$. Then
\begin{equation}\label{pr2,1}
\frac 12\Big(m_{\alpha}(\mu)-\int\limits_{(0,1)}u^{\alpha}\,
\mu(du)\Big)\le -(1-\alpha)\int
\limits_{(0,1]}\frac{K_{\mu}(-x)\,dx}{x^{1+\alpha}}\le
c(\mu)\alpha^{-1}m_{\alpha}(\mu),
\end{equation}  
where $c(\mu):=1/\int_{\mathbb R_+}\frac{\mu(du)}{1+u}$.

Moreover, $m_{\alpha}(\mu)<\infty$ with $\alpha\in(0,1)$
if and only if
\begin{equation}\label{pr2,1*}
-\int\limits_{(0,1]}\frac{ K_{\mu}(-x)\,dx}{x^{1+\alpha}}<\infty.
\end{equation}
\end{proposition}
\begin{proof}
In the first step we shall prove the right-hand side of \eqref{pr2,1}.
Without loss of generality we assume that $m_{\alpha}(\mu)<\infty$.
Since $1+\psi_{\mu}(-x)\ge \frac 1{c(\mu)}$ for $x\in(0,1]$,
we have
\begin{equation}\notag
 -K_{\mu}(-x)\le-c(\mu)\psi_{\mu}(-x)\le c(\mu)\Big(x\int\limits_{[0,1/x)}u\,\mu(du)
+\mu([1/x,\infty))\Big),\quad x\in(0,1].
\end{equation}
Taking into account that $m_{\alpha}(\mu)=\alpha\int\limits_{\mathbb R_+}
x^{\alpha-1}\mu([x,\infty))\,dx$, we finally obtain
\begin{equation}\label{pr2,1a}
  \begin{aligned}
-\frac 1{c(\mu)}\int\limits_{(0,1]}\frac{ K_{\mu}(-x)\,dx}{x^{1+\alpha}}&
\le\int\limits_{(0,1]}x^{-\alpha}
\int\limits_{[0,1/x)}u\,\mu(du)\,dx
+\int\limits_{(0,1]}x^{-1-\alpha}\mu([1/x,\infty))\,dx\\
&\le \int\limits_{[1,\infty)}u\int\limits_{(0,1/u]}x^{-\alpha}\,dx\,
\mu(du)+\frac 1{1-\alpha}\int\limits_{[0,1)}u\,
\mu(du)+\frac{m_{\alpha}(\mu)}{\alpha} \\
& \le
\frac{m_{\alpha}(\mu)}{\alpha(1-\alpha)}.
  \end{aligned}
\end{equation}
Let us prove the left-hand side of \eqref{pr2,1}, assuming without loss of
generality that $\int\limits_{(0,1]}\frac{K_{\mu}(-x)\,dx}{x^{1
+\alpha}}<\infty$.
Since, for $x>0$,
\begin{equation}\notag
 -K_{\mu}(-x)\ge-\psi_{\mu}(-x)\ge
\frac 12 x\int\limits_{[0,1/x)}u\,\mu(du),
\end{equation}
we have the lower bound 
\begin{equation}
\label{pr2,1b}
\begin{aligned}
-\int\limits_{(0,1]}\frac{K_{\mu}(-x)\,dx}{x^{1+\alpha}}&\ge \frac 12 \int
\limits_{(0,1]}x^{-\alpha}\int\limits_{[0,1/x)}u\,\mu(du)\,dx\ge \frac 12
\int\limits_{[1,\infty)}u\int\limits_{(0,1/u]} x^{-\alpha}\,dx
\,\mu(du)\\
&=\frac 1{2(1-\alpha)}\Big(m_{\alpha}(\mu)-\int
\limits_{(0,1)}u^{\alpha}\,\mu(du)\Big).
\end{aligned}
\end{equation}

The inequalities \eqref{pr2,1} follow from \eqref{pr2,1a} and \eqref{pr2,1b}.

Finally, statement \eqref{pr2,1*} is a direct consequence of \eqref{pr2,1}.

\end{proof}

\begin{lemma}\label{l2.1*}
Let $\mu_1$ and $\mu_2$ be probability measures from $\bf M_+$ such that  
$m_p(\mu_1)<\infty$ and $m_p(\mu_2)<\infty$ for some $p\in\mathbb N$. 
Then $m_p(\mu_1\boxtimes\mu_2)<\infty$. 
\end{lemma}
\begin{proof} 
By Theorem~\ref{th2.3}, there exist $Z_1(z)$ and $Z_2(z)$ 
from the~class $\mathcal K$ such that (\ref{2.8}) holds. 
By Proposition~\ref{pr2.1}, 
\begin{equation}\label{l.1}
K_{\mu_j}(-x)=-r_1(\mu_j) x+r_2(\mu_j) x^2+\dots+(-1)^pr_p(\mu_j) x^{p}+o(x^{p})
\quad\text{for}\quad x>0
\quad\text{and}\quad x\downarrow 0, 
\end{equation}
where $r_1(\mu),r_2(\mu),\dots,r_p(\mu)$ are the boolean cumulants.
Hence
\begin{equation}\label{l.1*}
K_{\mu_j}(Z_j(-x))=r_1(\mu_j) Z_j(-x)+r_2(\mu_j) Z_j^2(-x)+\dots
+r_p(\mu_j) Z_j^{p}(-x)+o(Z_j^{p}(-x))
\end{equation}
for $x>0,\,x\downarrow 0$ and $j=1,2$. From the~first relation 
of (\ref{2.8}) we conclude that, for the~same $x$,
\begin{equation}\notag
Z_j(-x)=-r_1(\mu_k)x+o(x),\qquad j,k=1,2,\quad j\ne k. 
\end{equation}
Let us assume that there exist real numbers 
$t_1^{(j)},t_2^{(j)},\dots,t_m^{(j)},\,j=1,2,\,m\le p-1$, such that
\begin{equation}\label{l.2}
Z_{j}(-x)=t_1^{(j)}x+t_2^{(j)}x^2+\dots+t_m^{(j)}x^{m}+o(x^{m})
\quad\text{for}\quad x>0
\quad\text{and}\quad x\downarrow 0. 
\end{equation}
Then from the~first relation of (\ref{2.8}) and from (\ref{l.1*}), 
(\ref{l.2}) we conclude that 
\begin{align}\label{l.3}
Z_j(-x)&=-r_1(\mu_k)x+r_2(\mu_k) xZ_k(-x)+\dots+(-1)^{p}r_p(\mu_k) xZ_k^{p-1}(-x)
+o(xZ_j^{p-1}(-x))\notag\\
&=t_1^{(j)}x+\dots+t_{m}^{(j)}x^{m}+t_{m+1}^{(j)}x^{m+1}+o(x^{m+1})
\end{align}
for real numbers $t_1^{(j)},t_2^{(j)},\dots,t_{m}^{(j)},t_{m+1}^{(j)},\,j=1,2$, and for 
$x>0,\,x\downarrow 0$.
Thus, induction may be used and (\ref{l.3}) holds for $m=p$. 
Since $K_{\mu_j}(Z_j(-x))=K_{\mu_1\boxtimes\mu_2}(-x),x>0$,
we easily obtain the~assertion of the~lemma from (\ref{l.1*}), (\ref{l.3}) with $m=p$ and 
from Proposition~\ref{pr2.1}. 
\end{proof}



\begin{lemma}\label{l2.2}
Let $\mu_1$ and $\mu_2$ be probability measures from $\bf M_+$ such that  
$m_{\alpha}(\mu_1)<\infty$ and $m_{\beta}(\mu_2)<\infty$, 
where $0<\alpha,\beta\le 1$. 
Then $m_{\alpha\beta}(\mu_1\boxtimes\mu_2)<\infty$.
\end{lemma}
\begin{proof} 
If the assumptions of the lemma hold with $\alpha=\beta=1$ the assertion of 
the lemma follows from
Lemma~\ref{l2.1*}. 

Consider the case, where the assumptions of the lemma hold with 
$0<\alpha<\beta$ and $\beta=1$.
By Theorem~\ref{th2.3}, 
there exist $Z_1(z)$ and $Z_2(z)$ 
from the~class 
$\mathcal K$ such that (\ref{2.8}) holds. By Proposition~\ref{pr2.1} and (\ref{pr3*}), 
\begin{equation}\notag
K_{\mu_1}(-x)=-m_1(\mu_1)x(1+o(1))
\end{equation} 
for positive $x$ such that 
$x\downarrow 0$. Hence 
\begin{equation}\notag
K_{\mu_1}(Z_1(-x))=m_1(\mu_1)Z_1(-x)(1+o(1))
\end{equation} 
for the~same $x$ and, by (\ref{2.8}), 
we have 
\begin{equation}\notag
Z_2(-x)=-m_1(\mu_1)x(1+o(1)).
\end{equation}  
From this relation and Proposition~\ref{pr2.2} we conclude that
\begin{equation}\notag
-\int\limits_{(0,x_0]}x^{-1-\alpha}K_{\mu_2}(Z_2(-x))\,dx\le
-\int\limits_{(0,x_0]}x^{-1-\alpha}K_{\mu_2}(-2m_1(\mu)x)\,dx<\infty,
\end{equation}
where $x_0$ is a sufficiently small positive constant.
Since $K_{\mu_2}(Z_2(-x))=K_{\mu_1\boxtimes\mu_2}(-x)$, by
Proposition~\ref{pr2.2}, we arrive at the~assertion 
of the~lemma for $\alpha\in(0,1)$ and $\beta=1$.

Consider the case, where the assumptions of the lemma hold with
$0<\alpha,\beta<1$. 

As above, by Theorem~\ref{th2.3}, there exist $Z_1(z)$ 
and $Z_2(z)$ from the~class 
$\mathcal K$ such that (\ref{2.8}) holds. By Proposition~\ref{pr2.2}, we have
\begin{equation}\label{l2,1}
-\int\limits_{(0,1]}x^{-1-\alpha}K_{\mu_1}(-x)\,dx
\le\frac{c(\mu_1)m_{\alpha}(\mu_1)}{\alpha(1-\alpha)}
\quad\text{and}\quad
-\int\limits_{(0,1]}x^{-1-\beta}K_{\mu_2}(-x)\,dx
\le\frac{c(\mu_2)m_{\beta}(\mu_2)}{\beta(1-\beta)},
\end{equation}
where $c(\mu_j),\,j=1,2$, are constants defined in Proposition~\ref{pr2.2}.
We obtain from~(\ref{2.8}) the~relation, for $x>0$,
\begin{equation}\label{l2,2} 
K_{\mu_1}(Z_1(-x))=K_{\mu_2}(Z_2(-x)).
\end{equation}
Recalling (\ref{2.5}) we deduce from (\ref{l2,2}) that, for $x\in(0,x_0]$ with sufficiently small $x_0>0$,
\begin{align}\label{l2,3} 
-\frac 12 \psi_{\mu_1}(Z_1(-x))&\le -\psi_{\mu_2}(Z_2(-x))
\le -Z_2(-x)\int\limits_{[0,-1/Z_2(-x))}
u\,\mu_2(du)+\mu_2([-1/Z_2(-x),\infty))\notag\\
&\le 2m_{\beta}(\mu_2)(-Z_2(-x))^{\beta}.
\end{align}
Since, by (\ref{2.8}),
\begin{equation}\notag 
-Z_2(-x)=xK_{\mu_1}(Z_1(-x))/Z_1(-x))\le 2x\psi_{\mu_1}(Z_1(-x))/Z_1(-x)),\quad x\in(0,x_0],
\end{equation}
we get from (\ref{l2,3}) the~bound
\begin{equation}\label{l2,4} 
-\frac 12(\psi_{\mu_1}(Z_1(-x))/Z_1(-x))^{1-\beta}Z_1(-x)
\le 2^{1+\beta}m_{\beta}(\mu_2)x^{\beta},
\quad x\in(0,x_0].
\end{equation}
On the~other hand $f(x):=\psi_{\mu_1}(Z_1(-x))/Z_1(-x)$ is 
a positive strictly monotone function
such that $\lim_{x\to 0} f(x)$ is not equal to 0. Hence we obtain 
from (\ref{l2,4}) that
\begin{equation}\label{l2,5}
-Z_1(-x)\le c(\mu_1,\mu_2)m_{\beta}(\mu_2)x^{\beta},\quad x\in(0,x_0],
\end{equation}
where $c(\mu_1,\mu_2)$ is a~positive constant depending on $\mu_1$ and $\mu_2$ only. 
It remains to note,
using (\ref{l2,1}), that
\begin{align}
&-\int\limits_{(0,x_0]}x^{-1-\alpha\beta}K_{\mu_1}(Z_1(-x))\,dx\le
-\int\limits_{(0,x_0]}x^{-1-\alpha\beta}K_{\mu_1}(-c(\mu_1,\mu_2)
m_{\beta}(\mu_2)x^{\beta})\,dx\notag\\
&\le c(\mu_1,\mu_2,\alpha,\beta)(m_{\beta}(\mu_2))^{\alpha}<\infty,\notag
\end{align}
where $c(\mu_1,\mu_2,\alpha,\beta)$ is a~positive constant depending 
on $\mu_1,\mu_2,\alpha$, and $\beta$ 
only. By Proposition~\ref{pr2.2}, the~lemma is proved.
\end{proof}

\begin{proposition}\label{pr3.3*}
Let $T$ and $S$ be free random variables such that $m_{p/2}(\mu_{T^2}\boxtimes\mu_{S^2})<\infty$
with some $p>0$. Then $TS\in L^p(\mathcal A,\tau)$ and $\tau(|TS|^p)=m_{p/2}(\mu_{T^2}\boxtimes\mu_{S^2})<\infty$.
\end{proposition}
\begin{proof} 
Since the distribution of $\abs{T}S^2\abs{T}$ is $\mu_{T^2}\boxtimes\mu_{S^2}$, 
we have
\begin{equation}\notag
\tau((\abs{T}S^2\abs{T})^{p/2})=m_{p/2}(\mu_{T^2}\boxtimes\mu_{S^2})<\infty. 
\end{equation}
Using the polar decomposition $T=u\abs{T}$, where $u\in{\mathcal A}$ is a unitary element,
we obtain
\begin{align}
\tau(\abs{TS}^{p})&
=\tau((u\abs{T}S^2\abs{T}u^*)^{p/2})=
\tau(u(\abs{T}S^2\abs{T})^{p/2}u^*)\notag\\
&=\tau((\abs{T}S^2\abs{T})^{p/2})
=m_{p/2}(\mu_{T^2}\boxtimes \mu_{S^2})<\infty.\notag
\end{align}
The proposition is proved.
\end{proof}

\begin{lemma}\label{l2.3}
Let $\mu_1,\mu_2,\dots,\mu_n$ be probability measures from $\bf M$ such that  
$\rho_1(\mu_1\boxplus\mu_2\boxplus\dotsm\boxplus\mu_n)<\infty$. 
Then $\rho_1(\mu_1)<\infty,\dots,
\rho_1(\mu_n)<\infty$.
\end{lemma}
\begin{proof}
  It suffices to prove the case $n=2$, the general case follows by induction.

  Let $T_1,T_2$ be free random variables with distributions 
  $\mu_1,\mu_2$,
  respectively, such that $\tau(\abs{T_1+T_2})<\infty$. 
  Without loss of generality we may assume that the distributions
  $\mu_1$, $\mu_2$ are not point masses.
  For, if $T$ is any measurable operator
  and $\lambda$ is any constant, 
  by a simple application of the Minkowski inequality
  we have
  $\norm{T}_1<\infty$ if and only if $\norm{T+\lambda I}_1<\infty$.
  By the same argument
  we can ensure that the spectra of each
  $T_1$ and $T_2$ are not contained in either the positive
  or the negative real axis.
  
  Consider the~projections $p_{T_1}^{(t)}=e_{T_1}([0,t])$ 
  and $p_{T_2}^{(t)}=e_{T_2}([0,t])$,
  where $t>0$ and $e_{T_1},e_{T_2}$ are an $\mathcal A$-valued
  spectral measures on $\mathbb R$, which are countably additive in 
  the weak$^*$ topology on $\mathcal A$.
  By our assumptions, these projections are nonzero for sufficiently
  large $t$.
  Set $T_j^{(t)}=p_{T_1}^{(t)}T_jp_{T_2}^{(t)}$ for $j=1,2$
  and note that by (\ref{au1.1}),
  \begin{equation}\label{l3,1}
  \tau(\bigabs{T_1^{(t)}+T_2^{(t)}})\le\tau(\abs{T_1+T_2})<\infty.
  \end{equation}
  On the~other hand, since the~random variables $p_{T_1}^{(t)}T_1$ 
  and $T_2 p_{T_2}^{(t)}$ are bounded,
  using freeness of the~corresponding random variables, 
  we have 
  $\tau(T_1^{(t)}+T_2^{(t)})
  =\tau(p_{T_1}^{(t)}T_1)\,\tau(p_{T_2}^{(t)})
    +\tau(p_{T_1}^{(t)})\,\tau(T_2 p_{T_2}^{(t)})$ and we obtain from (\ref{l3,1}) that
  \begin{equation}\notag
    \tau(p_{T_2}^{(t)})
    \int\limits_{[0,t]}u\,\mu_1(du)
    +
    \tau(p_{T_1}^{(t)})
    \int\limits_{[0,t]}u\,\mu_2(du)
    \le\tau(\abs{T_1+T_2})<\infty
  \end{equation}
  and in the limit  $t\to\infty$ this implies
  both
  \begin{equation}\notag
    \int\limits_{[0,\infty)}u\,\mu_1(du)<\infty
    \qquad
    \int\limits_{[0,\infty)}u\,\mu_2(du)<\infty.
  \end{equation}
  In the~same way we prove that
  \begin{equation}\notag
  \int\limits_{(-\infty,0)}u\,\mu_1(du) < \infty
    \qquad
    \int\limits_{(-\infty,0)}u\,  \mu_2(du)<\infty.
  \end{equation}
  Thus we have proved that $\rho_1(\mu_j)
  =\int\limits_{(-\infty,\infty)}u\,  \mu_j(du)<\infty$
  for $j=1,2$.
\end{proof}
\begin{proposition}\label{pr3.3}
Let $\{T_j\}_{j=1}^k$ be a~family of free elements in $\tilde{A}_{sa}$ 
such that
\begin{equation}\notag
\tau(\abs{T_j}^s)<\infty\quad\text{for all}\quad
s\in\mathbb N\quad\text{and}\quad j=1,2,\dots,k.
\end{equation}
Then $\tau(T_{j_1}T_{j_2}\dotsm T_{j_n})=0$
whenever $\tau(T_{j_l})=0,\,l=1,2,\dots,n$, and all alternating
sequences $j_1,j_2,\dots,j_n$ of $1$'s, $2$'s, and $k$'s, 
i.e., $j_1\ne j_2\ne\dots\ne j_n$. 
\end{proposition}

This proposition is well-known. In particular one can obtain a proof
using arguments of the paper by Bercovici and Voiculescu~\cite{BeVo:1993}.

\section{Proofs of the main results}

In order to prove Theorems~\ref{th2.1} and \ref{th2.2} we need the~following lemma.

\begin{lemma}\label{lem5.1}
Let $(T_j)_{j\in I}$ be free random variables 
in $W^*$-probability space $(\mathcal A,\tau)$
such that $\tau(\abs{T_j}^d)<\infty$ for some $d\in\{2,3,\dots\}$
and any $j\in I$.
Then
\begin{equation}
\tau(\abs{T_{k_1}^{n_1}T_{k_2}^{n_2}\dotsm T_{k_s}^{n_s}})<\infty
\end{equation}
for any choice of indices $k_1\ne k_2\ne\dots\ne k_s,\,s\ge 2$
and any choice of strictly positive integers $n_1,n_2,\dots,n_s$ 
such that $n_1+n_2+\dots+n_s=d+1$. 
\end{lemma}
\begin{proof}
We may assume that the distributions of $\abs{T_i}$ are not point
masses, otherwise the concerning operator is bounded
and the conclusion is trivial.
Let $d=2p,\,p\in\mathbb N$. Then we write 
$T_{k_1}^{n_1}T_{k_2}^{n_2}\dotsm T_{k_s}^{n_s}=
T_{k_1}^{n_1-1}(T_{k_1}T_{k_2})T_{k_2}^{n_2-1}\dotsm T_{k_s}^{n_s}$. 
By Lemma~\ref{l2.1*} and by Proposition~{\ref{pr3.3*}, we have
\begin{align}\label{5.8}
\tau(\abs{T_{k_1}T_{k_2}}^{d})
=m_p(\mu_{T_{k_1}^2}\boxtimes \mu_{T_{k_2}^2})<\infty.
\end{align}
Applying the~H\"older inequality (\ref{au1.3}), we easily obtain, 
using (\ref{5.8}),
\begin{align}\label{5.9}
\tau(&\abs{T_{k_1}^{n_1-1} (T_{k_1}T_{k_2})T_{k_2}^{n_2-1}\dotsm T_{k_s}^{n_s}})\notag\\
&\le(\tau(\abs{T_{k_1}}^d))^{(n_1-1)/d}(\tau(\abs{T_{k_1}T_{k_2}}^d))^{1/d}
(\tau(\abs{T_{k_2}}^d))^{(n_2-1)/d}
\dotsm(\tau(\abs{T_{k_s}}^d))^{n_s/d}<\infty.
\end{align}

Let $d=2p+1,\,p\in\mathbb N$. Consider first the case $s=2$, i.e.,
terms of the form $T_{k_1}^{n_1}T_{k_2}^{n_2}$ with
$k_1\ne k_2$ and $n_1+n_2=d+1$. 
By the~assumptions of the lemma, we see that
$m_{d/(2n_1)}(\mu_{T_{k_1}^{2n_1}})<\infty$ and 
$m_{d/(2n_2)}(\mu_{T_{k_2}^{2n_2}})<\infty$.

If $n_1=n_2=p+1$, then $\frac d{2n_1}=\frac d{2n_2}=1-\frac 1{2p+2}$ and 
$\Big(1-\frac 1{2p+2}\Big)^2>\frac 12$. By Lemma~\ref{l2.2}, we conclude that
$m_{1/2}(\mu_{T_{k_1}^{2n_1}}\boxtimes\mu_{T_{k_2}^{2n_2}})<\infty$ and then, by Proposition~\ref{pr3.3*},
we have
\begin{equation}\label{5.10}
\tau(\abs{T_{k_1}^{n_1}T_{k_2}^{n_2}})=
m_{1/2}(\mu_{T_{k_1}^{2n_1}}\boxtimes\mu_{T_{k_2}^{2n_2}})<\infty.  
\end{equation}

If $1\le n_1<p+1$ and $p+1<n_2\le 2p+1$, then $\frac d{2n_1}>1$ 
and $\frac 12\le\frac d{2n_2}<1$. 
By Lemma~\ref{l2.2}, $m_{1/2}(\mu_{T_{k_1}^{2n_1}}\boxtimes\mu_{T_{k_2}^{2n_2}})<\infty$ and,
by Proposition~\ref{pr3.3*}, we have the~relation~(\ref{5.10}) again. 

Consider now terms of the~form $T_{k_1}^{n_1}T_{k_2}T_{k_3}^{n_3}$ with
$k_1\ne k_2,\,k_3\ne k_2$ and $n_1+n_3=d$. Let for definiteness $n_1\le p$ and $n_3\ge p+1$.
Since $m_1(\mu_{T_{k_2}^2})<\infty$ and
$m_{d/(2n_3)}(\mu_{T_{k_3}^{2n_3}})<\infty$, 
we get, by
Lemma~\ref{l2.2}, that $m_{d/(2n_3)}(\mu_{T_{k_2}^2}
\boxtimes\mu_{T_{k_3}^{2n_3}})<\infty$.
By Proposition~\ref{pr3.3*}, we see that $\tau(\abs{T_{k_2}T_{k_3}^{n_3}}^{d/n_3})
=m_{d/(2n_3)}(\mu_{T_{k_2}^2}\boxtimes\mu_{T_{k_3}^{2n_3}})<\infty$.
Then, using the~H\"older inequality (\ref{au1.3}), we obtain
\begin{equation}\label{5.11a}
\tau(\abs{T_{k_1}^{n_1}T_{k_2}T_{k_3}^{n_3}})\le(\tau(\abs{T_{k_1}^{n_1}}^{d/n_1}))^{n_1/d}
(\tau(\abs{T_{k_2}T_{k_3}^{n_3}}^{d/n_3}))^{n_3/d}<\infty.
\end{equation}
 
Now consider a term of the~form $T_{k_1}^{n_1}T_{k_2}^{n_2}T_{k_3}^{n_3}$ with
$k_1\ne k_2\ne k_3$ and $n_1+n_2+n_3=d+1,\,n_1\ge 1,n_2\ge 2,n_3\ge 1$. 
Rewrite it in the~form
\begin{equation}\notag
T_{k_1}^{n_1}T_{k_2}^{n_2}T_{k_3}^{n_3}
=T_{k_1}^{n_1-1}(T_{k_1}T_{k_2})T_{k_2}^{n_2-2}(T_{k_2}T_{k_3})
T_{k_3}^{n_3-1}
\end{equation}
and note that as in the proof of (\ref{5.8}) we have
\begin{equation}\label{5.11}
\tau(\abs{T_{k_1}T_{k_2}}^{d-1})=\tau(\abs{T_{k_1}T_{k_2}}^{2p})
=m_p(\mu_{T_{k_1}^2}\boxtimes \mu_{T_{k_2}^2})<\infty
\end{equation}
and similarly
\begin{equation}\label{5.12}
\tau(\abs{T_{k_2}T_{k_3}}^{d-1}) <\infty.
\end{equation}
Now in view of (\ref{5.11}) and (\ref{5.12}), we deduce
with the help of the~H\"older inequality (\ref{au1.3}) 
\begin{align}\label{5.13}
\tau(\abs{T_{k_1}^{n_1}T_{k_2}^{n_2}T_{k_3}^{n_3}})
&\le(\tau(\abs{T_{k_1}}^{d-1}))^{\frac{n_1-1}{d-1}}
(\tau(\abs{T_{k_1}T_{k_2}}^{d-1}))^{\frac 1{d-1}}\notag\\
&\times(\tau(\abs{T_{k_2}}^{d-1}))^{\frac{n_2-2}{d-1}}(\tau(\abs{T_{k_2}T_{k_3}}^{d-1}))^{\frac 1{d-1}}
(\tau(\abs{T_3}^{d-1}))^{\frac{n_3-1}{d-1}}
<\infty.
\end{align}
Now  for any positive integers $k_1\ne k_2\ne\dots\ne k_s,\,s\ge 4$, 
and any positive integers $n_1,n_2,\dots,n_s$ such that $n_1+n_2+\dots+n_s=d+1$
we can write
\begin{equation}\notag
T_{k_1}^{n_1}T_{k_2}^{n_2}T_{k_3}^{n_3}T_{k_4}^{n_4}\dotsm T_{k_s}^{n_s}=
T_{k_1}^{n_1-1}(T_{k_1}T_{k_2})T_{k_2}^{n_2-1}T_{k_3}^{n_3-1}
(T_{k_3}T_{k_4})T_{k_4}^{n_4-1}
\dotsm T_{k_s}^{n_s}.
\end{equation} 
Repeating  the~previous arguments, we easily obtain 
\begin{align}\label{5.14}
\tau(\abs{T_{k_1}^{n_1}T_{k_2}^{n_2}T_{k_3}^{n_3}T_{k_4}^{n_4}
\dotsm T_{k_s}^{n_s}})&\le(\tau(\abs{T_{k_1}}^{d-1}))^{\frac{n_1-1}{d-1}}
(\tau(\abs{T_{k_1}T_{k_2}}^{d-1}))^{\frac 1{d-1}}(\tau(\abs{T_{k_2}}^{d-1}))^{\frac{n_2-1}{d-1}}\notag\\
&\times(\tau(\abs{T_{k_3}}^{d-1}))^{\frac{n_3-1}{d-1}}(\tau(\abs{T_{k_3}T_{k_4}}^{d-1}))^{\frac 1{d-1}}
(\tau(\abs{T_{k_4}}^{d-1}))^{\frac{n_4-1}{d-1}}\notag\\
&\times(\tau(\abs{T_{k_5}}^{d-1}))^{\frac{n_5}{d-1}}\dotsm
(\tau(\abs{T_{k_s}}^{d-1}))^{\frac{n_s}{d-1}}<\infty. 
\end{align}

The assertion of the~lemma follows from (\ref{5.9})--(\ref{5.11a}), (\ref{5.13}) and (\ref{5.14}).
}\end{proof}

{\bf Proof of Theorem~$\ref{th2.2}$}.
We need to prove that under the assumptions of Theorem~$\ref{th2.2}$
if the~forms $L$ and $Q$ are free, then $\tau(\abs{T_1}^s)<\infty$ for 
all $s\in\mathbb N$.

Consider the~free elements $L$ and $Q$ of the~probability 
space $(\mathcal A,\tau)$. 

In the~first step we shall prove that $\tau(\abs{T_1}^3)<\infty)$. 
Write the~relation
\begin{align}\label{5.3}
QL=\sum_j a_{jj}b_jT_j^3+\sum_{j\ne k}(a_{jj}b_kT_j^2T_k+a_{jk}b_kT_jT_k^2)
+\sum_{j\ne k,\,k\ne l}a_{jk}b_lT_jT_kT_l.
\end{align}

By the~Minkowski inequality (\ref{au1.4}), we see that 
\begin{equation}\notag
(\tau(L^2))^{1/2}
\le\sum\abs{b_j}\tau(\abs{T_j}^2)^{1/2}<\infty. 
\end{equation}
Since, by (\ref{au1.3}), $(\tau(\abs{T_jT_k}))^2
\le \tau(T_j^2)\tau(T_k^2)<\infty,j,k=1,2,\dots,n$,
we have, by the~Minkowski inequality (\ref{au1.4}) again, 
\begin{equation}\notag
\tau(\abs{Q})\le\sum_{j,k}^n\abs{a_{jk}}\tau(\abs{T_jT_k})\le \sum_{j}^n\abs{a_{jj}}\tau(\abs{T_j}^2)
+\sum_{j\ne k}^n\abs{a_{jk}}(\tau(\abs{T_j}^2))^{1/2}\,(\tau(\abs{T_k}^2))^{1/2}<\infty. 
\end{equation}
This means that $L$ has finite second moment
and $Q$ has finite first moment.

Since $\abs{QL}^2=QL^2Q$, we note that
$\mu_{\abs{QL}^2}=\mu_{Q^2}\boxtimes\mu_{L^2}$ and $\tau(\abs{QL})=
m_{1/2}(\mu_{Q^2}\boxtimes\mu_{L^2})$. Noting that, 
$m_{1/2}(\mu_{Q^2})<\infty$ and
$m_1(\mu_{L^2})<\infty$, by Lemma~\ref{l2.2}, we arrive at the~inequality
$m_{1/2}(\mu_{\abs{QL}^2})<\infty$. Hence, by Proposition~\ref{pr3.3*}, $\tau(\abs{QL})<\infty$. 

By Lemma~\ref{lem5.1}, we have the following bounds
\begin{equation}\label{5.3*}
\tau(\abs{T_kT_j^2})<\infty,\,\tau(\abs{T_k^2T_j})<\infty,\,j\ne k,\quad\text{and}
\quad \tau(\abs{T_jT_kT_l})<\infty,\,j\ne k\ne l.
\end{equation}



Return to (\ref{5.3}). Using the~Minkowski inequality (\ref{au1.4})
and (\ref{5.3*}) we obtain from (\ref{5.3}) that
\begin{align}\label{5.4}
\tau
(
\Bigabs{\sum_j a_{jj}b_jT_j^3}
)
&\le \tau(\abs{QL})+\sum_{j\ne k}\abs{b_k}
\bigl(
\abs{a_{jj}}\tau(\abs{T_j^2T_k})+\abs{a_{jk}}\tau(\abs{T_jT_k^2})
\bigr)\notag\\
&+\sum_{j\ne k,\,k\ne l}\abs{a_{jk}b_l}\tau(\abs{T_jT_kT_l})<\infty.
\end{align}
By Lemma~\ref{l2.3}, we conclude from this bound that 
$\tau(\abs{T_1}^3)<\infty$ as was to be proved.

Now assume that $\tau(\abs{T_j}^{d})<\infty$ for $d\ge 3$. 
We have, by the~Minkowski inequality (\ref{au1.4}) 
that
\begin{equation}\notag
(\tau(\abs{L}^{d}))^{1/d}\le \sum_{j}\abs{b_j}(\tau(\abs{T_j}^d))^{1/d}<\infty. 
\end{equation}
In addition, for $p=3,4$, we have, by Lemma~\ref{l2.2} and Proposition~\ref{pr3.3*},
\begin{equation}\notag
\tau(\abs{T_jT_k}^{p/2})=m_{p/4}(\mu_{T_j^2}\boxtimes\mu_{T_k^2})<\infty.
\end{equation}
Therefore, for $p=3,4$,
\begin{equation}\notag
\tau(\abs{Q}^{p/2})\le \sum_{j}^n\abs{a_{jj}}(\tau(\abs{T_j}^p))^{2/p}+
\sum_{j\ne k}^n\abs{a_{jk}}(\tau(\abs{T_jT_k}^{p/2}))^{2/p}<\infty,
\end{equation}
if $\tau(\abs{T_j}^p)<\infty$ for $p=3,4$, respectively.



Let $d=3$. In view of the~inequalities 
$m_{3/4}(\mu_{Q^2})<\infty$
and $m_{3/4}(\mu_{L^4})<\infty$, 
by Lemma~\ref{l2.2}, 
we arrive at the~inequality $m_{9/16}(\mu_{\abs{QL^2}^2})=m_{9/16}
(\mu_{Q^2}\boxtimes\mu_{L^4})<\infty$. Therefore, by Proposition~\ref{pr3.3*},
$\tau(\abs{QL^2})<\infty$.  

Let $d\ge 4$. 
Since 
$m_{1}(\mu_{Q^2})<\infty$ and
$m_{1/2}(\mu_{L^{2(d-1)}})<\infty$, by Lemma~\ref{l2.2}, 
we arrive at the~inequality
$m_{1/2}(\mu_{Q^2}\boxtimes\mu_{L^{2(d-1)}})<\infty$. 
Hence, by Proposition~\ref{pr3.3*}, $\tau(\abs{QL^{d-1}})=m_{1/2}(\mu_{Q^2}\boxtimes\mu_{L^{2(d-1)}})<\infty$.

Consider the~relation
\begin{equation}\label{5.7*}
QL^{d-1}=
\sum_j a_{jj}b_j^{d-1}T_j^{d+1}+\sum_{s=2}^{d+1}
\sum \alpha_{k_1k_2\dotsm k_s}T_{k_1}^{n_1}T_{k_2}^{n_2}\dotsm T_{k_s}^{n_s},
\end{equation}
where the~summation in sum of the~second summand 
on the~right-hand side of (\ref{5.7*})
is taken over all positive integers $k_1\ne k_2\ne\dots\ne k_s$ 
such that $k_j=1,2,\dots,n$,
and any positive integers $n_1,n_2,\dots,n_s$ such that $n_1+n_2+\dots+n_s=d+1$, 
and $\alpha_{k_1k_2\dotsm k_s}$ are real coefficients.

By Lemma~\ref{lem5.1}, we see that, for the considered values of 
$k_j$ and $n_j$,
\begin{equation}\label{5.8*}
\tau(\abs{T_{k_1}^{n_1}T_{k_2}^{n_2}\dotsm T_{k_s}^{n_s}})<\infty.
\end{equation}
Using the~Minkowski inequality (\ref{au1.4})
and (\ref{5.8*}) we obtain from (\ref{5.7*}) that
\begin{equation}\notag
\tau(\Bigabs{\sum_j a_{jj}b_j^{d-1}T_j^{d+1}})\le\tau(\abs{QL^{d-1}}
+\sum_{s=2}^{d+1}\sum \abs{\alpha_{k_1k_2\dotsm k_s}}\tau(\abs{T_{k_1}^{n_1}
T_{k_2}^{n_2}\dotsm T_{k_s}^{n_s}})<\infty. 
\end{equation}
Now, by Lemma~\ref{l2.3}, we conclude that $\tau(\abs{T_1}^{d+1})<\infty$.

Thus, induction may be used and the~theorem is proved.
$\square$

{\bf Proof of Theorem~$\ref{th2.1}$}.
Let the free random variables $T_1,T_2,\dots,T_n$ satisfy the assumptions of Theorem~\ref{th2.1}.
Then, as it is easy to see, the free random variables $T_1,T_2,\dots,T_n$ satisfy the assumptions of
Theorem~\ref{th2.2} as well. By this theorem $\tau(\abs{T_j}^k)<\infty,\,k\in\IN,\,j=1,2,\dots,n$.
Noting that the~arguments of the~paper~\cite{Le:2003} hold
for free identically distributed random variables with
finite moments of all order, we obtain the desired result
repeating step by step these arguments (see \cite{Le:2003}, p. 416--418). 
$\square$

\begin{thebibliography}{99}
\itemsep=\smallskipamount

\bibitem{Akh:1965} Akhiezer,~N. I. 
{\em The~classical moment problem and some related questions 
in analysis.}
Hafner, New York (1965).

\bibitem{BelBe:2007} Belinschi,~S. T. and Bercovici,~H.
\emph{A new approach to subordination results in free probability.}
J. Anal. Math.  {\bf 101}, 357--365 (2007).

\bibitem{Benaych:2006:Taylor}
Benaych-Georges,~F., \emph{Taylor expansions of {$R$}-transforms:
  application to supports and moments}, Indiana Univ. Math. J. \textbf{55}
  (2006), no.~2, 465--481.


\bibitem{BeVo:1993} Bercovici,~H., and Voiculescu,~D.
\emph{Free convolution of measures with unbounded support.}
Indiana Univ. Math. J., {\bf 42}, 733--773 (1993).


\bibitem{Bi:1998} Biane,~Ph.
\emph{Processes with free increments.}
Math. Z., {\bf 227}, 143--174 (1998).

\bibitem{ChG:2005} Chistyakov,~G.~P. and G\"otze,~F.
\emph{The~arithmetic of distributions in free probability theory.}
ArXiv: math/0508245.

\bibitem{HiPe:2000} Hiai,~F. and Petz,~D. 
{\em The semicircle law, free random variables and entropy.}
American Mathematical Society, (2000).

\bibitem{HiNaYo:1999} Hiwatashi,~O., Nagisa,~M. and Yoshida,~H.
\emph{The characterizations of a semicircle law by the~certain 
freeness in a $C^*$-probability space.}
Probab. Theory Relat. Fields,
{\bf 113}, 115--133 (1999).

\bibitem{KLR:1973} Kagan,~A.~A., Linnik,~Yu.~V. and Rao,~C.~R. 
{\em Characterization problems in mathematical statistics.}
John Wiley $\&$ Sons, New York, London, Sydney, Toronto (1973).

\bibitem{KS:1949} Kawata,~T. and Sakamoto,~H. 
\emph{On the characterization of the independence of the sample mean and 
the sample variance.}
J. Math. Soc. Japan. {\bf 1}, 111-115 (1949).

\bibitem{Kr:1977} Krein,~M.~G. and Nudel'man,~A.~A.
{\em  The Markov moment problem and extremal problem.}
Amer. Math. Soc., Providence, Rhode Island (1977).


\bibitem{Le:2003} Lehner,~F., 
\emph{Cumulants in noncommutative probability theory II.}
Probab. Theory Relat. Fields,
{\bf 127}, 407--422 (2003).

\bibitem{Le:2004} Lehner,~F., 
\emph{Cumulants in noncommutative probability theory I.
Noncommutative exchangeability systems.}
Probab. Theory Relat. Fields,
{\bf 248}, 67--100 (2004).

\bibitem{Ma:1992} Maassen,~H., 
\emph{Addition of freely independent random variables.}  
J. Funct. Anal. {\bf 106}, no. 2, 409--438 (1992).

\bibitem{MuNe:1936} Murray,~H. and von Neumann,~J.
\emph{On rings of operators.} 
Ann. of Math. (2)  {\bf 37}, no. 1, 116--229 (1936).
\bibitem{NiSp:2006} Nica,~A. and Speicher,~R.
{\em Lectures on the combinatorics of free probability.}
Cambridge University Press, (2006).

\bibitem{SpWo:1997} Speicher,~R. and Woroudi,~R.
{\em Boolean convolutions.}
Free probability theory (Waterloo, ON, 1995),  267--279, 
Fields Inst. Commun., 12, Amer. Math. Soc., Providence, RI, 1997. 


\bibitem{Ta:2003} Takesaki,~M. 
{\em Theory of operator algebras II.}
Springer-Verlag Berlin Heidelberg New York (2003).

\bibitem{Z:1951} Zinger,~A.~A. 
\emph{On independent samples from a normal population.}
Uspekhi Matem. Nauk, {\bf 6}, 172 (1951).

\bibitem{Vo:1986} Voiculescu,~D.~V.
\emph{ Addition of certain noncommuting random variables.}
J. Funct. Anal., {\bf 66}, 323--346 (1986).

\bibitem{Vo:1987} Voiculescu,~D.~V.
\emph{Multiplication  of certain noncommuting random variables.}
J. Operator Theory, {\bf 18}, 223--235 (1987). 

\bibitem{Vo:1992} Voiculesku,~D.~V., Dykema,~K., and Nica,~A. 
{\em Free random variables.}
CRM Monograph Series, No 1, A.M.S., Providence, RI (1992).




\end{thebibliography}

          






\end{document}